\documentclass[11pt,  reqno]{amsart}

\usepackage{amsmath,amssymb,amscd,amsthm,amsxtra, esint}

\usepackage[english, french]{babel}

\allowdisplaybreaks[2]

\newtheorem{theorem}{Theorem} [section]

\newtheorem{lemma}[theorem]{Lemma}
\newtheorem{proposition}[theorem]{Proposition}

\newtheorem*{acknowledgment}{Acknowledgment}

\DeclareMathOperator*{\supp}{supp}

\newcommand{\I}{\hspace{0.5mm}\text{I}\hspace{0.5mm}}

\newcommand{\noi}{\noindent}
\newcommand{\Z}{\mathbb{Z}}
\newcommand{\R}{\mathbb{R}}

\newcommand{\T}{\mathbb{T}}

\let\Im=\undefined\DeclareMathOperator*{\Im}{Im}

\newcommand{\al}{\alpha}
\newcommand{\be}{\beta}
\newcommand{\dl}{\delta}

\newcommand{\eps}{\varepsilon}

\newcommand{\g}{\gamma}
\newcommand{\G}{\Gamma}

\newcommand{\cj}{\overline}
\newcommand{\dx}{\partial_x}
\newcommand{\dt}{\partial_t}

\renewcommand{\G}{\mathcal{G}}
\renewcommand{\I}{\mathcal{I}}

\newcommand{\E}{\mathcal{E}}

\newcommand{\too}{\longrightarrow}

\numberwithin{equation}{section}
\numberwithin{theorem}{section}

\begin{document}
\selectlanguage{english}



\title[GWP of DNLS on the circle]
{A remark on global well-posedness
of the derivative nonlinear Schr\"odinger equation on the circle
}

\author{Razvan Mosincat and Tadahiro Oh}

\address{
School of Mathematics\\
The University of Edinburgh, 
and The Maxwell Institute for the Mathematical Sciences\\
James Clerk Maxwell Building\\
The King's Buildings\\
Peter Guthrie Tait Road
Edinburgh\\
EH9 3FD, United Kingdom}
\email{r.o.mosincat@sms.ed.ac.uk}


\email{hiro.oh@ed.ac.uk}

\subjclass[2010]{35Q55}

\keywords{derivative nonlinear Schr\"odinger equation; global well-posedness;
Gagliardo-Nirenberg inequality}


\begin{abstract}

In this note, we consider the derivative nonlinear Schr\"odinger equation
on the circle.
In particular, by 
adapting Wu's recent argument to the periodic setting,
we prove its global well-posedness in $H^1(\T)$, provided that 
the mass is less than $4\pi$.
Moreover, this mass threshold is independent of spatial periods.
\end{abstract}

\maketitle

%
%
%

%

\section{Introduction}

In this note, we consider global well-posedness of the following
derivative nonlinear Schr\"odinger equation (DNLS) 
on $\T_L := \R/(L\Z)\simeq [0, L)$:
\begin{align}
\begin{cases}
i \dt u + \dx^2 u = i \dx (|u|^2 u)\\
u|_{t = 0} = u_0 \in H^1(\T_L), 
\end{cases}
\qquad (x, t) \in \T_L\times \R.
\label{DNLS1}
\end{align}

\noi
The equation \eqref{DNLS1} is known to be completely integrable
and thus possesses an infinite sequence of 
conservation laws.
For our analysis, the following conservation laws play an important role:
\begin{align}
 & \text{Mass:}   & M(u)   & = \int_{\T_L} |u|^2 dx, \label{C1}\\
 & \text{Hamiltonian:}& H(u) &  = \Im \int_{\T_L} u \cj{u}_x dx+ \frac{1}{2} \int_{\T_L} |u|^4 dx, \label{C2}\\
 & \text{Energy:} & E(u)  & = \int_{\T_L} |u_x|^2 dx
+\frac{3}{2} \Im \int_{\T_L} u u \cj{u u}_x dx
+
\frac{1}{2}\int_{\T_L}|u|^6 dx.
\label{C3}
\end{align}

Let us briefly go over the known well-posedness
results on $\T$, i.e.~with $L = 1$.
Herr \cite{Herr} proved local well-posedness of \eqref{DNLS1} in $H^\frac{1}{2}(\T)$.
He also proved global well-posedness in $H^1(\T)$, 
under the assumption that the mass is less than $\frac{2}{3}$.\footnote{As pointed out in \cite[Remark 6.1]{Herr}, 
 this mass threshold $\frac{2}{3}$ is not sharp.
 In view of the corresponding result \cite{HO} on $\R$, 
 it is likely that the mass threshold can be improved to $2\pi$ within the framework of \cite{Herr}.}
In the low regularity setting, Win \cite{Win} applied the $I$-method \cite{CKSTT1, CKSTT2} and proved
global well-posedness of \eqref{DNLS1}
in $H^{s}(\T)$, $s>\frac{1}{2}$, 
provided that mass 
 is sufficiently small.\footnote{In \cite{Win}, 
the mass threshold was not quantified in a precise manner.
See, for example, \cite[Lemma 3.4]{Win}.}
 Our main interest in this note is 
to 
 improve the mass threshold
 for  global well-posedness of \eqref{DNLS1}
 in the smooth setting, i.e.~in $H^1(\T_L)$.

On $\R$, Hayashi-Ozawa
\cite{HO} proved global well-posedness of \eqref{DNLS1} in $H^1(\R)$, 
provided that mass is less than $2\pi$.
By the sharp Gagliardo-Nirenberg inequality
due to Weinstein \cite{W}:
\[ \|f\|_{L^6(\R)} \leq \frac{4}{\pi^2} \|\dx f \|_{L^2(\R)}^\frac{1}{3}\|f\|_{L^2(\R)}^\frac{2}{3}, \]

\noi
this smallness of  mass guarantees that 
the energy $E(u)$ remains coercive and 
controls the $\dot H^1(\R)$-norm of a solution.
Thus, this situation is analogous to 
that for the focusing quintic nonlinear Schr\"odinger equation (NLS).\footnote{Note that both
DNLS and the focusing quintic NLS   on $\R$ are mass-critical.}
On the one hand,  there is a dichotomy between global well-posedness
and finite time blowup solutions for the focusing quintic NLS on $\R$, 
where the mass threshold is given by the mass of the ground state.
On the other hand, DNLS has a much richer structure
such as complete integrability
and the question of global well-posedness/finite time blowup solutions for large masses 
has been open
for decades.
Recently, Wu \cite{Wu1, Wu2} made a progress in this direction.
In particular, he proved global well-posedness
of \eqref{DNLS1} on $\R$
for masses less than $4 \pi$.
Our main result states that global well-posedness of \eqref{DNLS1}
in the periodic setting also  holds with the same mass threshold $4\pi$.

\begin{theorem}\label{THM:1}
Let $L > 0$. Then, the derivative nonlinear Schr\"odinger equation \eqref{DNLS1}
on $\T_L$ is globally well-posed in $H^1(\T_L)$, 
provided that the mass is less than $4\pi$.
\end{theorem}

\noi
Theorem \ref{THM:1} improves the known mass threshold in \cite{Herr}
for global well-posedness in $H^1(\T)$.
Moreover, note that the mass threshold $4\pi$ is independent of the period $L$.



The question of global well-posedness/finite time blowup solutions
for larger masses ($\geq 4\pi$) remains open 
 on  both $\R$ and $\T_L$.
It is worthwhile to note that 
\eqref{DNLS1} possesses
finite time blowup solutions
under the Dirichlet boundary condition
on intervals and the half line $\R_+ = [0, \infty)$, 
if $E(u) < 0$ (under some extra conditions).
See \cite{Tan, Wu1}.

The proof of  Theorem \ref{THM:1}
is based on Wu's argument \cite{Wu2}. 
On the one hand, the following 
sharp Gagliardo-Nirenberg inequality:
\begin{align}
\|f\|_{L^6(\R)}\leq C_\text{GN} \|\dx f\|_{L^2(\R)}^\frac{1}{9} \|f\|_{L^4(\R)}^\frac{8}{9}
\label{GN0}
\end{align}
	
\noi
plays an important role in \cite{Wu2}.
Here, the optimal constant $C_\text{GN}$ is given by 
 $C_\text{GN} = 3^\frac{1}{6} (2\pi)^{-\frac 19}$.
See Agueh \cite{A}.
On the other hand, 
 \eqref{GN0} does not hold  on $\T_L$
and thus we need to consider a  variation of \eqref{GN0}
suitable for our application on $\T_L$.
Moreover, the gauge transform 
in the periodic setting introduces
extra terms in the conservation laws
that we need to control.

\section{Proof of Theorem \ref{THM:1}}

In this section, we present the proof of Theorem \ref{THM:1}.
Note that Theorem \ref{THM:1} follows
once we prove the following proposition for all sufficiently small  $\dl > 0$.

\begin{proposition}\label{PROP:1}
Let $L, \dl > 0$. Then,  \eqref{DNLS1}
on $\T_L$ is globally well-posed in $H^1(\T_L)$
provided that the mass is less than $4\pi\big(1 + \frac{2\dl}{5L}\big)^{-2}$.
\end{proposition}

\noi
The remaining part of this note is devoted to the proof of Proposition \ref{PROP:1}. 
%
%

We first establish 
the following version of the Gagliardo-Nirenberg inequality on $\T_L$
which incorporates the sharp constant from \eqref{GN0}. 
The proof is a simple adaptation of  the argument in Lebowitz-Rose-Speer \cite{LRS}.

\begin{lemma}\label{LEM:GN}
Let $\dl > 0$.
Then, we have 
\begin{align}
\| f\|_{L^6(\T_L)}
\leq C_\textup{GN}\bigg(1 + \frac{2\dl}{5L}\bigg)^\frac{2}{9}
\Big(\|\dx f\|^2_{L^2(\T_L)} + \frac 2{\dl L^\frac{1}{2}} \|f\|^2_{L^4(\T_L)}\Big)^\frac{1}{18}
\|f\|^\frac{8}{9}_{L^4(\T_L)}.
\label{GN1}
\end{align}
\end{lemma}

\noi
for $f \in H^1(\T_L) $.	

\begin{proof}
%
%
%
Let $f \in H^1(\T_L) \subset C(\T_L)$. 
By periodicity,  we assume that 
\begin{align}
|f(0)| = |f(L)| \leq L^{-\frac{1}{4}}
\|f\|_{L^4(\T_L)}
\label{GN2}
\end{align}

\noi
without loss of generality.
Let $F$ be an extension of $f$ on $[0, L]$ to $\R$
such that (i) $\supp F \subset [- \dl, L + \dl]$
and (ii) $F$ linearly interpolates $0$ and $f(0)$
on $[-\dl, 0]$ and $f(L)$ and $0$ on $[L, L+\dl]$.
Then, by a direct calculation, we have
\begin{align}
\| f\|_{L^6(\T_L)}^6 & \leq \|F\|_{L^6(\R)}^6,\label{GN3} \\
\|F\|_{L^4(\R)}^4 
& \leq \|f\|_{L^4(\T_L)}^4 + \frac {2\dl}{5}|f(0)|^4
\leq \bigg(1 + \frac{2\dl}{5L} \bigg) \|f\|_{L^4(\T_L)}^4,
\label{GN4}
\\
\|\dx F\|_{L^2(\R)}^2 
& \leq \|\dx f\|_{L^2(\T_L)}^2 + 2\frac{|f(0)|^2}{\dl}
\leq \|\dx f\|_{L^2(\T_L)}^2+ \frac{2}{\dl L^\frac{1}{2}} \|f\|_{L^4(\T_L)}^2.
\label{GN5}
\end{align}

\noi
Then, the desired estimate \eqref{GN1}
follows from \eqref{GN0}
with
\eqref{GN3}, \eqref{GN4}, and \eqref{GN5}.
%
%
\end{proof}

Next, we briefly go over the gauge transform associated to \eqref{DNLS1}
with a general parameter $\be \in \R$.
The gauge transform for DNLS was first introduced by Hayashi-Ozawa \cite{HO}
in the non-periodic setting.
Herr \cite{Herr} adapted the gauge transform (with $\be = 1$) to the periodic setting,
exhibiting remarkable cancellations of certain resonances.

Given $f \in H^1(\T_L)$, let $\I(f)$ denotes
the mean-zero antiderivative of $|f|^2$.  
Then, we define $\mathcal{G}_\be: H^1(\T_L) \to H^1(\T_L)$
by  $\G_\be (f) := e^{-i\be \I(f)} f.$
With a slight abuse of notations, 
we also use $\G_\be$ to denote a map$: C([-T, T]: H^1(\T_L)) \to C([-T, T]: H^1(\T_L))$
by  
\[\G_\be (u) := e^{-i\be \I(u)} u.\]

\noi
Given    a local-in-time solution  $u \in C([-T, T]: H^1(\T_L))$ to \eqref{DNLS1}, 
the conservation of mass allows us to  define    
\[\mu  =  \mu(u) := \frac{1}{L} M(u)
= \frac{1}{L} \int_{\T_L} |u|^2 dx, \]

\noi 
independent of time.
We then define 
 \begin{equation}
 v(x, t) : = \G^\be (u)(x, t)  = \G_\be(u)(x - 2 \be \mu t, t),
\label{Gauge1}
 \end{equation}

\noi
A straightforward computation shows that $v$ satisfies 
\begin{align}
i \dt v  +  \dx^2 v = 2 (1 - \be) i  |v|^2 v_x +  (1-2\be) i  v^2 \cj{v}_x
+ \be \mu |v|^2 v + \be (\tfrac{1}{2} - \be) |v|^4 v
- \psi(v) v, 
\label{DNLS2} 
\end{align}

\noi
where
\[ 
 \psi(v):  = \frac{\be}{L} \bigg(\int_{\T_L} 2 \Im  (v \cj{v}_x) + 
\Big(\frac{3}{2} - 2\be\Big) |v|^4\bigg) v  
+ \be^2 \mu^2 .\]

\noi
It follows from \eqref{Gauge1} that $M(v)$ is conserved for \eqref{DNLS2}.
Moreover, the conservation laws $H(u)$ and $E(u)$ in \eqref{C2} and \eqref{C3}
for \eqref{DNLS1} yield the following conservation laws
for \eqref{DNLS2}:
\begin{align}
H(v) & = \Im \int_{\T_L} v \cj{v}_x dx+ \bigg(\frac{1}{2} - \be\bigg) \int_{\T_L} |v|^4 dx
+ L \be  \mu^2, \label{C4}\\
E(v) & = \int_{\T_L} |v_x|^2 dx
+\bigg( \frac{3}{2} -2 \be\bigg) \Im \int_{\T_L}  vv \cj{v v}_x dx
+
\bigg( \be^2 - \frac{3}{2}\be + \frac{1}{2}\bigg)\int_{\T_L}|v|^6 dx  \notag \\
& \hphantom{X}
+ 2\be \Im \int_{\T_L} v \cj{v}_x dx+ \be \bigg(\frac{3}{2} - 2\be\bigg) \mu \int_{\T_L} |v|^4 dx
+ L \be^2 \mu^3.
\label{C5}
\end{align}

\noi
See, for example, the computations in \cite{NORS}.
It is worthwhile to note  that $H(v)$ is not a Hamiltonian for \eqref{DNLS2} in general.
%
%
%
%
In establishing well-posedness, 
the gauge transform with $\be = 1$ played an important role
\cite{HO, Herr, Win}.
For our purpose, we set $\be = \frac 34$ in the following
so that the second term in \eqref{C5} is not present, 
and let $\G : = \G^{\frac 34}$.
In particular, 
it follows from 	
 \eqref{C4} and \eqref{C5} 
with the conservation of $\mu = \mu(v):= L^{-1} M(v)$
that 
%
the following quantity 
\begin{align}
\mathcal{E} (v) 
:= \int_{\T_L} |v_x|^2 dx
- \frac{1}{16}\int_{\T_L}|v|^6 dx  
+ \frac{3}{8} \mu \int_{\T_L}  |v|^4 dx.
\label{C6}
\end{align}

\noi
is conserved for \eqref{DNLS2},
where $v = \G(u)$. 



Now, we move onto the proof of Proposition \ref{PROP:1}.
The proof follows closely to that in \cite{Wu2}.
By time reversibility, we restrict our attention to positive times.
For notational simplicity,  
we suppress the domain of integration $\T_L$
with the understanding that all the norms are taken over $\T_L$.
First, 
recall  that  Herr's local well-posedness result \cite{Herr}
yields a simple blowup alternative;
either (i) the solution $u$ to \eqref{DNLS1} exists globally 
or (ii) there exists a finite time $T_* $
such that $\lim_{t \uparrow T_*}\|u(t)\|_{\dot H^1} = \infty$.

Fix $\dl > 0$. We argue by contradiction.
Suppose that there exists a solution $u$ to \eqref{DNLS1}
such that 
$M(u) < 4\pi\big(1 + \frac{2\dl}{5L}\big)^{-2}$
but
$\lim_{t \uparrow T_*}\|u(t)\|_{\dot H^1} = \infty$
for some finite time $T_* > 0$.
Let  $v = \G (u)$ be the corresponding solution to \eqref{DNLS2}.
Since 
the gauge transform $\G$ in \eqref{Gauge1} is 
continuous on $C([-T, T]: H^1)$, 
our assumption  implies that there exists a sequence $\{t_n\}_{n\in \mathbb{N}} \subset \R_+$
such that 
$\lim_{n \to \infty }\|v(t_n)\|_{\dot H^1} = \infty$
while 
$M(v) = M(u) < 4\pi\big(1 + \frac{2\dl}{5L}\big)^{-2}$.
Then, it follows 
from the conservation of $\mathcal{E}(v)$ that 
\begin{align}
	\| v(t_n)\|_{L^6} \to \infty, 
\label{bound0}
\end{align}

\noi
as $n \to \infty$.

As in \cite{Wu2}, 
we define $\{f_n\}_{n\in \mathbb{N}}$ by 
\[ f_n = \frac{\| v(t_n)\|_{L^4}^4}{
\| v(t_n)\|_{L^6}^3}.\]

\noi
Then, we have the following lemma.

\begin{lemma}\label{LEM:bound}
Let $L, \dl > 0$.
Then, we have 
\begin{align}
2 C_\textup{GN}^{-\frac 92} 
\bigg(1 + \frac{2\dl}{5L}\bigg)^{-1}
+ \eps_n 
\leq f_n \leq M(v)^\frac{1}{2}, 
\label{bound1}
\end{align}
\noi
where $\eps_n = \eps_n (L, \dl) \to 0$ as $n \to \infty$.
In particular, 
$\| v(t_n)\|_{L^4} \to \infty$ as $n \to \infty$.
\end{lemma}

\begin{proof}
The upper bound in \eqref{bound1}
follows from H\"older's inequality.
Then,  it follows from 
the upper bound in \eqref{bound1}
and 
\eqref{bound0} that 
\begin{align}
\g_n : = 
 \bigg(\frac{2}{\dl L^\frac{1}{2}} - \frac{3}{8}\mu \| v(t_n)\|_{L^4}^2\bigg)
  \frac{\| v(t_n)\|_{L^4}^2}{
\| v(t_n)\|_{L^6}^6} \too 0,
\label{bound2}
\end{align}

\noi 
as $n \to \infty$.
By Lemma \ref{LEM:GN} with \eqref{C6}, we have 
\begin{align}
f_n 
& \geq 
 C_\textup{GN}^{-\frac{9}{2}}\bigg(1 + \frac{2\dl}{5L}\bigg)^{-1}
\Big(\|\dx v(t_n)\|^2_{L^2} + \frac 2{\dl  L^\frac{1}{2}} \|v(t_n)\|^2_{L^4}\Big)^{-\frac{1}{4}}
\|v(t_n)\|^\frac{3}{2}_{L^6}\notag \\
& = 2 C_\textup{GN}^{-\frac 92} 
\bigg(1 + \frac{2\dl}{5L}\bigg)^{-1}
\bigg( 1+
16
\frac{  \mathcal{E}(v)}{\|v(t_n)\|_{L^6}^{6} } + 16\g_n 
\bigg)^{-\frac{1}{4}}.
\label{bound3}
\end{align}

\noi
Then, 
the lower bound in \eqref{bound1} follows from 
\eqref{bound0}, \eqref{bound2}, and \eqref{bound3}
with the conservation of $\mathcal{E}(v)$.
The second claim follows from \eqref{bound0} and \eqref{bound1}.
\end{proof}

In the following, we use the conservation of 
the momentum $P(v)$ defined by
\begin{align*}
P(v) & := H(v) - \frac{3}{4L} M(v)^2 = \Im \int_{\T_L} v \cj{v}_x dx- \frac{1}{4} \int_{\T_L} |v|^4 dx.
\end{align*}

\noi
In order to exploit the momentum, 
we consider
modulated functions $\phi_n (x, t) = e^{i\al_n x} v(x, t)$
for some non-zero $\al_n \in 2\pi \Z/L$ (to be chosen later).
On the one hand, 
we have
\begin{align}
 P(v) +  \frac{1}{4} \int_{\T_L} |v|^4 dx
=  \Im \int_{\T_L} v \cj{v}_x dx
= - \frac{1}{2\al_n}\mathcal{E}(\phi_n)
+ \frac{\al_n}{2} M(v) + \frac{1}{2\al_n}\E(v).
\label{M1}
\end{align}
	
\noi
On the other hand, 
 by Lemma \ref{LEM:GN} with \eqref{C6} and \eqref{bound2}, 
we have
\begin{align}
\E\big((\phi_n(t_n)\big)
\geq - (\eta_n+\g_n) 
\|v(t_n)\|_{L^6}^6
\label{M2}
\end{align}

\noi
where $\eta_n$ is defined by 
\begin{align} \eta_n : = 
\frac{1}{16} -  \bigg(1 + \frac{2\dl}{5L}\bigg)^{-4}
 C_\textup{GN}^{-18} f_n^{-4}.
 \label{M2a}
\end{align}

\medskip

\noi
{\bf Case 1:} $\eta_n +\g_n \leq 0$ 
for infinitely many $n$.

In this case, we simply set $\al_n = \frac{2\pi}{L}$.
Then, 
for those values of $n$ with $\eta_n +\g_n \leq 0$, 
it follows 
from \eqref{M1} and \eqref{M2} with \eqref{bound2} that 
\begin{align*}
 \frac{1}{4} \|v(t_n) \|^4_{L^4}
& \leq    \frac{L}{4\pi}
(\eta_n +\g_n) \|v(t_n)\|_{L^6}^6
- P(v) + \frac{\pi}{L} M(v) + \frac{L}{4\pi}\E(v)\\
& \leq 
- P(v) + \frac{\pi}{L} M(v) + \frac{L}{4\pi}\E(v).
\end{align*}

\noi
Then, from  the conservation of $M$, $P$, and $\mathcal{E}$,
we conclude that 
$\|v(t_n) \|_{L^4} = O(1)$.
This is a contradiction to Lemma \ref{LEM:bound}.

\medskip

\noi
{\bf Case 2:}  $\eta_n + \g_n> 0$ 
for all sufficiently large $n$.

In this case, 
we choose 
\[\al_n 
:=  \frac{2\pi}{L}\Big[ \frac{L}{2\pi}  \big(M(v)^{-1}(\eta_n + \g_n)\big)^\frac{1}{2}
\| v(t_n)\|_{L^6}^3
\Big]
+ \frac{2\pi}{L} \in \frac{2\pi\Z}{L},\]

\noi
where 
$\g_n$ and $\eta_n$ are as in \eqref{bound2} and \eqref{M2a}.
Here,   $[x]$ denotes the integer part of $x$.
Then, 
from \eqref{M1} and \eqref{M2}, we have 
\begin{align*}
 \frac{1}{4}  \|v(t_n)\|_{L^4}^4
& \leq  
 \big(M(v)(\eta_n + \g_n)\big)^\frac{1}{2}
\|v(t_n)\|_{L^6}^3
- P(v) + \frac{\pi}{L} M(v) + \frac{1}{2\al_n}\E(v).
\end{align*}

\noi
Then, by
Lemma \ref{LEM:bound}, 
\eqref{bound0}, 
\eqref{bound2}, 
and \eqref{M2a}
along with the conservation of $M$, $P$, and $\mathcal{E}$, we obtain
\begin{align}
f_n^6
& \leq  
M(v)  f_n^4 
- 16 \bigg(1 + \frac{2\dl}{5L}\bigg)^{-4}
 C_\textup{GN}^{-18} M(v) + o(1)
\label{M3}
\end{align}

\noi
as $n\to \infty$.
Arguing as in \cite{Wu2}, 
we see that \eqref{M3} is impossible if
\[ M(u) = M(v) < 4\pi\bigg(1 + \frac{2\dl}{5L}\bigg)^{-2}.\]

\noi
This completes the proof of Proposition \ref{PROP:1}
and hence the proof of Theorem \ref{THM:1}.

\begin{acknowledgment} \rm
The authors would like to thank
Sebastian Herr and Yifei Wu
for  helpful comments. 
\end{acknowledgment}

\end{document}